\documentclass[12pt]{amsart}

\usepackage{graphicx}
\usepackage{amssymb}
\usepackage{amsmath}
\usepackage{amsthm}

\newtheorem{theorem}[]{Theorem}

\begin{document}

\title{On the Diophantine equation $\frac{1}{a} + \frac{1}{b} = \frac{q + 1}{pq}$. 
}
\author{Jeremiah W. Johnson}
\address{University of New Hampshire, Manchester, NH 03101, USA}
\email{jeremiah.johnson@unh.edu}
\urladdr{https://mypages.unh.edu/jwjohnson}

\begin{abstract}
Let $p$ and $q$ be distinct primes such that $q+1 | p-1$. In this paper we find all integer solutions $a$, $b$ to the equation $1/a + 1/b  = (q+1)/pq$ using only elementary methods.
\end{abstract}

\maketitle

\section{Introduction}\label{S:1}
Problem A1 of the 2018 William Lowell Putnam Mathematical Competition asked competitors to find all ordered pairs of positive integers $(a, b)$ for which $1/a + 1/b  = 3 / 2018$ \cite{putnam}. This is an example of a Diophantine equation of the form $1/a + 1/b = q+1/pq$, where $p$ and $q$ are distinct primes such that $q+1 | p-1$. In this note we generalize this problem and obtain all integer solutions $a, b$ to the equation $1/a + 1/b = q+1/pq$ for distinct primes $p, q$ such that $q+1 | p-1$ using only elementary arithmetic. 

\begin{theorem}
Let 
\begin{equation}\label{e:1}
    \frac{1}{a} + \frac{1}{b}=\frac{q+1}{pq},
\end{equation}
where $p, q$ are distinct primes such that $q+1 | p-1$. There are two trivial solutions to equation \ref{e:1}; namely $a = p, b=pq$, and $a = pq, b=p$. The remaining integer solutions of equation \ref{e:1} are given by the following formul\ae:
\begin{equation}\label{e:2}
a = \frac{\zeta p^2 + pq}{q+1}\text{,\hspace{2em}}b = \frac{\zeta pq + q^2}{\zeta(q+1)}
\end{equation}
and 
\begin{equation}\label{e:3}
a = \frac{\zeta pq + q^2}{\zeta(q+1)}\text{,\hspace{2em}}b = \frac{\zeta p^2 + pq}{q+1},
\end{equation}
where $\zeta\equiv 1 \mod q+1$ and $\zeta | q^2$.
\end{theorem}

\begin{proof}
Rewrite equation \ref{e:1} as $pq(a+b) = (q+1)ab$. Then $p | a$, $p | b$, or both. Consider first the case where $p | a$ and $p | b$. Write $a = pk$ and $b=pl$ for $k, l\in \mathbb{Z}^+$. Then $pq(pk + pl) = (q+1)p^2kl \implies q(k+l) = (q+1)kl$. If $k, l \geq 2$, this equation cannot be satisfied, so one of $k, l = 1$. Suppose $k=1$, so $a=p$, and $q(1+l)=(q+1)l\implies l=q$, leading to the trivial solutions $a=p, b=pq$, and, \emph{mutatis mutandis}, $a=pq, b=p$. \\

Suppose now that $p | a$ but $p \nmid b$. Again writing $a=pk$ for $k\in\mathbb{Z}^+$, we have $pq(pk + b) = (q+1)pkb \implies q(pk + b) = (q+1)kb$, or equivalently, $qpk = (qk + k - q)b$. Since $ p \nmid b$, it must be the case that $p | qk + k - q$. Write $qk + k - q = \zeta p$ for $\zeta \in \mathbb{Z}^+$. Then $qk + k = \zeta p + q \implies (q+1)k = \zeta(p-1) + \zeta + q$ Since $q+1 | p-1$, $\zeta \equiv 1 \mod q+1$. \\

Solving for $k$ and then $a$ and $b$ above leads to the formul\ae\hspace{0.1em} given in equation \ref{e:2}. To obtain solutions in $\mathbb{Z}$, it must be the case that $q+1 | \zeta p^2 + pq$ and both $\zeta$ and $q+1$ divide $\zeta pq + q^2$. The first condition is always satisfied given our earlier assumption that $q + 1 | p - 1$: writing $\zeta = \alpha(q+1) + 1$, we have $\zeta p^2 + pq = \alpha(q+1)p^2 + p^2 + pq = \alpha(q+1)p^2 + p(p + q)$, and $p + q = p - 1 + q + 1$ is divisible by $q+1$. It can similarly be shown that $q+1 | \zeta pq + q^2$. However, $\zeta | \zeta pq + q^2$ if and only if $\zeta | q^2$, leading to the result in Equation \ref{e:2}. If instead we assume that $p | b$ and $p \nmid a$, then, \emph{mutatis mutandis}, we arrive at the solutions given in Equation \ref{e:3}. 
\end{proof}

\bibliographystyle{acm}
\bibliography{diophantine.bib}

\end{document}